\documentclass[12pt]{amsart} %%
\usepackage[a4paper, margin = 2.5cm]{geometry}
\usepackage{enumerate,amsmath,amssymb,amsthm,bbm,hyperref}
\newcommand{\set}[1]{\left\{#1\right\}}
\newcommand{\norm}[1]{\left\lVert#1\right\rVert}
\newcommand{\bnorm}[1]{\big\lVert#1\big\rVert}
\newcommand{\sprod}[1]{\left\langle#1\right\rangle}
\newcommand{\pprod}[1]{\left(#1\right)}
\newcommand{\bprod}[1]{\bigl(#1\bigr)}
\newcommand{\ind}[1]{\mathbbm{1}_{#1}}

\theoremstyle{plain}\def\HK{\mathbb{H}_X}
\theoremstyle{plain}
\def\HKX{\mathcal{H}_X}

\newtheorem{theorem}{Theorem}[section]
\newtheorem{proposition}[theorem]{Proposition}
\newtheorem{corollary}[theorem]{Corollary}
\newtheorem{lemma}[theorem]{Lemma}
\theoremstyle{remark}

\newtheorem{remark}[theorem]{Remark}

\newcommand{\R}{\mathbb{R}}
\newcommand{\T}{\mathbb{T}}
\newcommand{\F}{\mathcal F}

\newcommand{\pr}{\mathrm{P}}
\newcommand{\ex}[1]{\mathrm{E}\left[\,#1\,\right]}
\newcommand{\bex}[1]{\mathrm{E}\bigl[\,#1\,\bigr]}

\newcommand{\eps}{\varepsilon}

\newcommand{\cY}{\CE{y}}
\newcommand{\CE}{}
\begin{document}

\title[Boundary non-crossing probabilities of Gaussian processes]{Boundary non-crossing probabilities of Gaussian processes: sharp bounds and asymptotics}
\author{Enkelejd Hashorva \and Yuliya Mishura \and Georgiy Shevchenko}

\begin{abstract}
We study boundary non-crossing probabilities 
$$
P_{f,u} := \pr\big(\forall t\in \T\ X_t + f(t)\le u(t)\big)
$$
for continuous centered Gaussian process $X$  indexed by some arbitrary compact separable metric space $\T$. We obtain both upper and lower bounds for $P_{f,u}$. The bounds are matching in the sense that they lead to precise logarithmic asymptotics for the large-drift case $P_{\cY f,u}$, $\cY\to+\infty$, which are two-term approximations (up to $o(\cY)$). The asymptotics are formulated in terms of the solution $\tilde f$ to the constrained optimization problem
$$
\norm{h}_{\HK}\to \min, \quad h\in \HK, h\ge f
$$
in the reproducing kernel Hilbert space $\HK$ of $X$. Several applications of the results are further presented. 
\end{abstract}



\maketitle
\section{Introduction}
In this article we are interested in boundary non-crossing probabilities
\begin{equation*} %\label{eq:pfu}
P_{f,u} := \pr\big(\forall t\in \T\ X_t + f(t)\le u(t)\big).
\end{equation*}
Here $X$ is a continuous centered Gaussian process defined on a compact separable metric space $\T$, $u\colon \T\to \R$ (boundary) and $f\colon \T\to \R$ (trend)  are some  deterministic functions\footnote{\CE{Further we impose stronger assumptions on the drift than on the boundary and also consider the asymptotics of $P_{yf,u}$ when $y\to+\infty$, and for this reason we do not write the probability in question  as $\pr\big(\forall t\in \T\ X_t \le g(t)\big)$ with $g(t) = u(t) - f(t)$.}}.  The continuity assumption is motivated by the observation that in order for the probability to be well defined, the corresponding event has to be generated by values of $X$ on  some countable subset of $\T$. Two most natural situations when this happens are the case of countable $\T$ (which we will  study elsewhere) and the case of a continuous process defined on a separable metric space, studied here. We further restrict ourselves to the more tractable case of compact $\T$ (and we also show how the case of locally compact $\T$ can be reduced to it). Sufficient conditions for continuity of Gaussian process are given in e.g.\ \cite[Chapter 10]{lifshits}.

Explicit formulas for $P_{f,u} $ are known only for very special $X$ and particular  $u,f$ with most prominent example $X$ being a Wiener process and $u,f$ being piece-wise linear functions, see e.g., \cite{KL98,NFR03,PW01}. In the absence of explicit formulas, several authors have obtained upper and lower bounds for the non-crossing probabilities of Gaussian processes with trend and/or  their asymptotic behavior. We list just few references on such question: Wiener process was considered in \cite{BHH07,BN05,Has05}; Brownian bridge, in \cite{BH05,BHHM03,BMHH03};  Brownian pillow and Brownian sheet in  \cite{Has10,BW}; additive Wiener field, in  \cite{HM14}; fractional Brownian motion, in \cite{HMS15}; \CE{ for closely related investigations, see 
	\cite{MR3608132,MR4001018}}.  

%Explicit formulas for $P_{f,u} $ are known only for very special $X$ and particular  $u,f$ with most prominent example $X$ being a Wiener process and $u,f$ being piece-wise linear functions, see e.g., \cite{KL98,NFR03,PW01}. In absence of explicit formulas, several authors have obtained upper and lower bounds for the non-crossing probabilities of Gaussian processes with trend and/or  their asymptotic behavior. We list just few references on such question: Wiener process was considered in \cite{BHH07,BN05,Has05}; Brownian bridge, in \cite{BH05,BHHM03,BMHH03};  Brownian pillow and Brownian sheet in  \cite{Has10,BW}; additive Wiener field, in  \cite{HM14}; fractional Brownian motion, in \cite{HMS15}. 

In the case where $\T=[0,T]$ the boundary non-crossing probabilities $P_{f,u}$ are related to survival probabilities
$$\pr\bigl(\forall t\in \T\ X_t+f(t)<u(t)\bigr)= \pr(\tau_u>T),$$ where $\tau_u = \inf\set{t\ge 0: X_t+f(t) \ge u(t)}$ is the hitting time of a moving boundary $u$. Such probabilities (typically their asymptotic behavior as $T\to\infty$) are studied in the now very active topic of persistence probabilities. We refer to \cite{aurzada} for a comprehensive review of the topic.

Under the continuity assumption, the process $X$ can be regarded as a centered Gaussian element in the separable Banach space (equipped with supremum norm)  $$C_0(\T;\T_0):=\set{g\in C(\T): \forall t\in \T_0 \ \ g(t) = 0 }$$
of continuous functions vanishing on the zero set $\T_0:= \set{t\in \T: X_t = 0\text{ a.s.}}$ of $X$. The zero set of the process is emphasized since the crucial role in asymptotic results is played by the injectivity of the covariance operator, which is defined on the dual space. In case of $C_0(\T;\T_0)$ its  dual is the space $M(\T_1)$ of signed finite measures on $\T_1 = \T\setminus \T_0$. If the process were considered as an element of $C(\T)$, the dual would be $M(\T)$, and the kernel of the covariance operator will contain the measures supported by $\T_0$. So such a setting is chosen to allow for the greatest generality (and note that $\T_0$ may be empty).

Our approach to getting bounds for $P_{f,u}$ is based on the change of measure with the help of Cameron--Martin theorem. For this reason we assume that  $P_{0,u} \in (0,1)$ and the drift $f$  belongs to the Cameron--Martin space (or reproducing kernel Hilbert space, RKHS)  $\mathbb{H}_X$ of $X$. 
The latter is defined in terms of the covariance function 
$$
R(t,s) = \ex{X_t X_s},\quad t,s\in \T
$$
as the completion of the space spanned by $R(t,\cdot)$ with respect to the scalar product defined as a linear extension of
$$
\bigl( R(t,\cdot),R(s,\cdot)\bigr)_{\mathbb H_X} = R(t,s).
$$
The Cameron--Martin space can be also described in terms of the covariance operator, defined by
$$
\sprod{\mathcal R \mu, \nu} = \ex{\sprod{X,\mu}\sprod{X,\nu}}, \mu,\nu\in M(\T_1),
$$
where $\sprod{\cdot,\cdot}$ denotes the duality pairing. 

A general lower bound for $P_{f, u}$ follows from \cite[Proposition 1.6]{aurzada-dereich}. For any $f\in \HK$, let $\tilde f$ be the \CE{metric} projection of zero to the closed convex set $C_f:=\{h \in \mathbb{H}_X, h\ge f\}$. Then, applying \cite[Proposition 1.6]{aurzada-dereich} with $S = C_f - f$ and the drift $f-\tilde f$, we get
\begin{equation}
\label{lB:gen}
P_{f,u} \ge  P_{ f- \tilde f, u} \exp \Bigl \{ -\frac{1} 2 {\bnorm{\tilde f}_{\mathbb{H}_X}^2 } - \bnorm{\tilde f}_{\mathbb{H}_X} \sqrt{-2  \log P_{ f- \tilde f, u}}  \Bigr\}.
\end{equation}	
We note in passing that comparable lower bounds to \eqref{lB:gen} follow also by \cite[Theorem 1.1']{KL98} or \cite[Theorem 7.3]{lifshits}.
From the above,  if further $P_{0,u} \in (0,1)$, then 
\begin{eqnarray}
\label{lower:LD}
\log P_{\cY f,u}  \ge  -\frac{\cY^2} 2 {\bnorm{\tilde f}_{\mathbb{H}_X}^2 }+ O(\cY), \quad  \cY\to \CE{+}\infty.
\end{eqnarray} 
The main (and a hard) problem is the derivation of an accurate  upper bound for $P_{\cY f,u}$ (which matches \eqref{lower:LD}) valid for all large $\cY$. One approach goes through the general large deviation principle for Gaussian measures, it is explained in detail in Subsection~\ref{sec:ldp}.

In this contribution we show that a sharp upper bound for $P_{f, u}$ can be determined if there exists a non-negative finite measure $\tilde \gamma\in M(\T_1)$ such that $\tilde f = \mathcal R \tilde \gamma\ge f$ and $\bprod{f-\tilde f,\tilde f}_{\HK}\ge 0$. In this case we establish in Theorem~\ref{thm:general} the following upper bound:
\begin{eqnarray*} % \label{eq:noncrossingprobupperbound-verygeneral}% \label{eq:thm1}
	P_{f,u} \le P_{0,u} \exp \Bigl \{- \frac{1}{2}\bnorm{\tilde f}_{\HK}^2 + \Theta(\CE{\tilde \gamma}, u) \Bigl \},
\end{eqnarray*}
where $$\Theta(\tilde \gamma, u) = \int_{\T_1} u(t)\tilde \gamma(dt),$$
and a similar lower bound. Under the additional assumption that the operator $\mathcal R$ is injective and some special assumption on $\HK$, we identify $\tilde f$ with the aforementioned projection and prove that
\begin{eqnarray}
\label{ldp}
\log P_{ \cY f,u} = -\frac{\cY^2}{2} \bnorm{\tilde f}^2_{\HK} + \cY \,   \Theta(\CE{\tilde \gamma}, u) + o(\cY), \quad \cY\to +\infty,
\end{eqnarray}
which implies an equality in \eqref{lower:LD} and further refines the asymptotics. 

In the special case where $X$ is a standard Wiener process,  $\tilde f$ is the least non-decreasing concave majorant of $f$, and the above asymptotics agrees with the known results for Brownian motion, see e.g.\ \cite{BHH07}.

%The asymptotics \eqref{ldp} is also closely related to the large deviation principle. Namely, denoting $X^u(t) = u(t) - X(t)$, $\eps = c^{-1}$ and setting $$A_f = \set{g\colon \T\to \R\mid \forall t\in \T\ g(t)\ge f(t)}$$ we can rewrite the boundary non-crossing probability as $P_{f,u} = \pr\bigl(\eps X^u \in A_f \bigr)$ so that the asymptotics \eqref{ldp} implies
%\begin{equation}\label{eq:false-ldp}
%\eps^2 \log \pr\bigl(\eps X^u \in A_f \bigr) \to -\frac12 \bnorm{\tilde f}^2_{\HK},\quad \eps\downarrow  0. 
%\end{equation}
%The claim in \eqref{eq:false-ldp}  looks very similar to the general large deviation principle for centered Gaussian measures due to Donsker and Varadhan \cite{DV76} (see also \cite[Theorem 4.5]{DGL96}). Indeed,  the right-hand side of \eqref{eq:false-ldp} features the infimum of the ``rate functional'' $\frac12 \norm{\cdot}_{\HK}$ over the ``target set'' $A_f$. However, the family of measures corresponding to $X^u$ does not satisfy the large deviation principle, for example, \eqref{eq:false-ldp} does not hold in general for $-A_f$ in place of $A_f$: the probability of corresponding event is typically zero. Moreover, to the best of our knowledge, there is no general large deviation principle of the form \eqref{eq:false-ldp} for non-centered Gaussian measures. 

The paper is organized as follows. The main results of the article are displayed  in Section 2, which starts with a brief introduction to Gaussian processes. In Subsection 2.1, we establish both upper and lower bounds for $P_{f,u}$.  The obtained results are then used to derive logarithmic asymptotics of non-crossing probabilities in Subsection 2.2. \CE{Subsection~\ref{sec:localcomp} shows how the results can be extended to the case of a locally compact parameter space. In Subsection~\ref{sec:ldp}, we show how an asymptotic upper bound matching~\eqref{LD:upper} can be derived using the general large deviation principle for Gaussian measures. In Section 3, we illustrate the  findings of Section 2 considering several important Gaussian processes. Section A contains some auxiliary results. }

%\section{General Gaussian processes}\label{sec:general}
\section{Main results} \label{sec:general}

Throughout the paper, $(\Omega,\F,\pr)$ is a complete probability space carrying all objects under consideration.

As in the Introduction,  $X=\set{X_t, t\in \mathbb T}$ is a continuous centered real-valued Gaussian process defined on some compact separable metric space $(\T, \tau)$, with  covariance function $R(t,s)$, 
which (thanks to continuity of $X$) is continuous in both $t$ and $s$. The zero set $\T_0 = \set{t\in \T: X_t = 0\text{ a.s.}}$ of $X$ is closed and can be given in terms of the covariance function:
%Below we present  some basics properties of  Gaussian processes, the details can be found in \cite{lifshits1,lifshits}.
$$
\T_0 = \set{t\in \T: R(t,t) = 0} =  \set{t\in \T: \forall s\in \T \ R(t,s) = 0 }.
$$
Recall that the process $X$ is a Gaussian element in the separable Banach space $C_0(\T;\T_0)$ of functions vanishing on $\T_0$, whose dual is $M(\T_1)$.
Hereafter  $\sprod{\cdot,\cdot}$  shall denote the duality pairing between $C_0(\T;\T_0)$ and $M(\T_1)$ i.e. 
$$\sprod{x,\mu} = \int_{\T} x(t) \mu(dt),\ x\in C_0(\T;\T_0), \mu\in M(\T_1)
$$ as well as between other spaces and their duals. We slightly abuse notation here since  $\mu$ is not defined on $\T_0$; there is no danger since $x(t)=0$ for $t\in \T_0$.

The covariance operator $\mathcal R\colon M(\T_1)\to C_0(\T;\T_0)$ corresponding to $X$ can be equivalently defined by
$$
\mathcal R \mu(t) = \int_{\T} R(t,s)\mu(ds), \mu\in M(\T_1),
$$
or
\begin{equation}\label{eq:R-sprod}
\sprod{\mathcal R \mu,\nu} =\ex{\sprod{X,\mu}\sprod{X,\nu}} =\int_{\T}\int_{\T} R(t,s)\mu(ds)\nu(dt), \mu,\nu\in M(\T_1).
\end{equation}
Since $R$ is a non-negative definite function, \eqref{eq:R-sprod}  defines an inner product on the quotient of $M(\T_1)$ modulo $\ker\mathcal R$. The completion of the latter with respect to this inner product is the Hilbert space of so-called measurable linear functionals, which will be denoted by $\mathcal H_X$, and the corresponding inner product will be denoted by $\pprod{\cdot,\cdot}_{\mathcal H_X}$. Moreover, thanks to \eqref{eq:R-sprod}, the operator $\mathcal R$ can be extended to $\mathcal H_X$ by continuity so that
$$
\pprod{\mu_1,\mu_2}_{\HKX} = \sprod{\mathcal R \mu_1,\mu_2}, \mu_1\in \HKX,\mu_2\in M(\T_1).
$$
Again, by continuity, the above duality pairing  can be extended to $\mu_1,\mu_2\in \mathcal H_X$. Similarly, by \eqref{eq:R-sprod}, $\sprod{X,\cdot}$ can be extended to an isometry between $\mathcal H_X$ and some subspace of $L^2(\Omega)$.  

It is also worth noting that for any $\mu\in \mathcal H_X$, the random variable $\sprod{X,\mu}$, being a mean square limit of centered Gaussian random variables, is a centered Gaussian random variable with variance $\bex{\sprod{X,\mu}^2} = \norm{\mu}^2_{\mathcal H_X}$.

\begin{remark}
	We slightly abuse rigor here, because the space $\mathcal H_X$ is  a completion of the quotient $M(\T_1)/\ker \mathcal R$, not a completion of $M(\T_1)$.
	For example, the book \cite{lifshits} goes through  $I^*\colon M(\T_1)\to M(\T_1)/\ker \mathcal R$. However, we decided to keep this slightly ambiguous notation for the sake of clarity and simplicity and in view of the fact that the main results of this article are formulated for the case where $\mathcal R$ is injective.
\end{remark}

Further, $\mathcal R$ defines an isometry between $\mathcal H_X$ and its image $\mathbb H_X = \mathcal R \mathcal H_X$ equipped with the inner product
$$ \pprod{\mathcal R \mu_1,\mathcal R \mu_2}_{\HK} = \pprod{\mu_1,\mu_2}_{\HKX}.
$$
Defining for $t\in \T$ the Dirac measure $\delta_t$ by $\sprod{x,\delta_t} = x(t)$, $x\in C_0(\T;\T_0)$, we have
\begin{eqnarray}
\label{covR}
\pprod{R(t,\cdot), R(s,\cdot)}_{\HK} =  \pprod{\mathcal R\delta_t, \mathcal R\delta_s}_{\HK} = \pprod{\delta_t,\delta_s}_{\HKX}  = R(t,s)
\end{eqnarray}
for all $t,s\in \T$, so the space $\HK$ is indeed the reproducing kernel Hilbert space (RKHS) of $X$, since it is unique with respect to the covariance reproducing property \eqref{covR}.

We present below  the classical Cameron--Martin theorem for $X$, see \cite[Theorem 5.1]{lifshits}. The formulation in \cite{lifshits} is given in terms of push-forward measures induced by $X$ and $X+f$ and is slightly different from the one given below, but it is easily seen to be equivalent.

%\begin{theorem}\label{thm:cm-general-pushforward}
%If $f = \mathcal R g\in \mathbb H_X$, then the measures $P^f$ and $P$ on $C_0(\T;\T_0)$, corresponding to $X + f$ and to $X$ respectively, are equivalent with
%$$
%\frac{dP^f}{dP}(x) = \mathcal{E}_{X}(f)(x) := \exp\Bigl\{ \sprod{x,g} - \frac12 \norm{g}_{\mathcal H_X}^2\Bigr\} = \exp\Bigl\{ \sprod{x,g} - \frac12 \norm{f}_{\mathbb H_X}^2\Bigr\}.
%$$
%\end{theorem}
%It may be more handful to consider the corresponding ``pushforward meaning''. Namely, denoting by $P$ the measure induced by $X$ on $C_0(\T;\T_0)$, for any $g\in \mathcal H_X$, we can consider the  element $\sprod{x,g}\in L^2(C_0(\T;\T_0),P)$ corresponding to $\sprod{X,g}$. 
\begin{lemma}[Cameron--Martin theorem]\label{thm:cm-general}
	If $f = \mathcal R \mu\in \mathbb H_X$, then the distribution of $X+f$ with respect to $\pr$ is the same as that of $X$ with respect to the measure $\pr^f$ with 
	$$
	\frac{d\pr^f}{d\pr} = \mathcal{E}_{X}(\mu) := \exp\Bigl\{ \sprod{X,\mu} - \frac12 \norm{\mu}_{\mathcal H_X}^2\Bigr\} = \exp\Bigl\{ \sprod{X,\mu} - \frac12 \norm{f}_{\mathbb H_X}^2\Bigr\}.
	$$
\end{lemma}

\subsection{Bounds for non-crossing probabilities} \label{subsec:bounds}

In this section, we study the boundary non-crossing probability
$$
P_{f,u}:= \pr\big(\forall t\in\T\ \ X_t + f(t)\le u(t)\big).
$$
Here $f\in\HK$ and $u\colon \T\to \R$ is a lower semicontinuous function such that $P_{0,u}>0$.
%\begin{remark}We recall that a function $u\colon \T \to \R$ is called lower semicontinuous if for any $t\in\T$
%	$$
%	\liminf_{s\to t}f(s)\ge f(t).
%	$$
%	A lower semicontinuous envelope, the largest lower semicontinuous function not exceeding $u$, is defined by $u_*(t)=\sup_{v\in C(\T): v\le u} v(t)$.
%	Using it, we can see that 
The assumption of lower semicontinuity of $u$ does not harm the generality. Indeed, in view of the continuity of $X$ and $f$, for any bounded function $u\colon \T\to \R$ we have 
$$
\set{\forall t\in\T\ \ X_t + f(t)\le u(t)} = \set{\forall t\in\T\ \ X_t + f(t)\le u_*(t)},
$$
where $u_*$ is the lower semicontinuous envelope of $u$. 
%\end{remark}

Further we derive lower and upper bounds for $P_{f,u}$ for any trend   $f\in \HK$.  
%These bounds are given in terms of elements
%$\tilde \mu\in \mathcal{H}_X$  and $\tilde f = \mathcal{R} \tilde \gamma$ for which we shall require the following conditions to hold, where we denote by $M^+(\T_1)$ the set of finite non-negative measures on $\T_1$:
%\begin{itemize}
%	\item[\rm (G1)] $\tilde \gamma\in M^+(\T_1)$;
%	\item[\rm (G2)] $ \sprod{f-\tilde f,\tilde \gamma}\ge 0$;
%	\item[\rm (G3)] $\tilde f\ge f$, i.e. $\tilde f(t)\ge f(t)$ for all $t\in\T$.
%\end{itemize}
%Note that (G2) may be equivalently expressed in terms of 
%$\bprod{f-\tilde f,\tilde f}_{\mathbb H_X}= \sprod{f-\tilde f,\tilde \gamma}$.
%
%We will show further that under some ``non-degeneracy'' assumptions on $\mathcal R$, a function  $\tilde f$ satisfying (G1)--(G3) is unique, and under some additional assumption it exists as a solution to certain constrained optimization problem. Also note that if $f =\mathcal R \gamma$ with
%$\gamma\in M^+(\T_1)$, then $\tilde f =f$ clearly satisfies (G1)--(G3) with $\tilde \gamma = \gamma$.
Denote by $M^+(\T_1)$ the set of finite non-negative measures on $\T_1$ and recall that $\Theta(\tilde \gamma,u) = \int_{\T_1}u(t)\tilde \gamma(dt)$.
\begin{theorem} \label{thm:general}
	Let $f\in \HK$ and suppose that $\tilde f = \mathcal R \tilde \gamma$ with $\tilde \gamma$ satisfying condition  
	\begin{itemize}
		\setlength{\itemindent}{3em}\setlength{\listparindent}{3.5em}\setlength{\parskip}{3em}
		\item[\rm (G1)] $\tilde \gamma\in M^+(\T_1)$.
	\end{itemize}
	
	\noindent 1. If the condition
	\begin{itemize}
		\setlength{\itemindent}{3em}\setlength{\listparindent}{3.5em}\setlength{\parskip}{3em}
		\item[\rm (G2)] $ \sprod{f-\tilde f,\tilde \gamma}\ge 0$
	\end{itemize}	
	is satisfied, then
	\begin{eqnarray} \label{eq:noncrossingprobupperbound-verygeneral}
	% \label{eq:thm1}
	P_{f,u} \le P_{f-\tilde f,u} \exp \Bigl \{- \frac{1}{2}\bnorm{\tilde f}_{\HK}^2 + \Theta(\tilde \gamma, u) \Bigr \}.
	\end{eqnarray}

	\noindent 2.
	Let  $u_-\colon \T\to \R$ be a continuous function such that   $u_{-}(t) < u(t)$ for all $t\in\T$. If further
	$$
	P_{0,u,u_-} := \pr\big(\forall t\in\T\ \ u_{-}(t)\le X_t\le u(t)\big)>0
	$$
	and  condition  
	\begin{itemize}
		\setlength{\itemindent}{3em}\setlength{\listparindent}{3.5em}\setlength{\parskip}{3em}
		\item[\rm (G3)] $\tilde f\ge f$, i.e. $\tilde f(t)\ge f(t)$ for all $t\in\T$
	\end{itemize}
	holds, then% with $\tilde f= \mathcal{R} \tilde\gamma$
	\begin{equation}\label{eq:lower}
	\begin{gathered}
	P_{f,u} \ge P_{0,u,u_{-}}  \exp \Bigl  \{ - \frac{1}{2} \bnorm{ \tilde f}_{\HK}^2 + \Theta(\tilde \gamma, u_{-}) \Bigr \}.
	\end{gathered}
	\end{equation}
	
\end{theorem}
\begin{proof} 1. 
	Using Lemma ~\ref{thm:cm-general} we have
	\begin{equation}\label{eq:cameron-martin}
	\begin{gathered}
	P_{f,u} = \ex{\ind{\forall t\in\T\ \ X_t+f(t)\le u(t)}} =
	\ex{\ind{\forall t\in\T\ X_t+f(t)-\tilde f(t)\le u(t)} \frac{d\pr^{
				\tilde f}}{d\pr}}\\
	%	= \ex{\ind{\forall t\in\T\ X_t+f(t)-\tilde f(t)\le u(t)} \mathcal{E}_{X}(\tilde \gamma)}
	= \ex{\ind{\forall t\in\T\ X_t+f(t)-\tilde f(t)\le u(t)} \exp \Bigl  \{ - \frac{1}{2} \bnorm{ \tilde f}_{\HK}^2 + \sprod{X,\tilde \gamma} \Bigr \}}.
	\end{gathered}
	\end{equation}
	Note that
	\begin{equation} \label{eq:integral-estimate}
	\sprod{X,\tilde \gamma} \le  \sprod{X+f-\tilde f,\tilde \gamma} = \int_{\T_1} \bigl(X_t + f(t)-\tilde f(t)\bigr)\tilde \gamma(dt) \le \int_{\T_1} u(t)\tilde \gamma(dt) = \Theta(\tilde \gamma,u)
	\end{equation}
	on $\set{{\forall t\in\T\ \ X_t+f(t)-\tilde f(t)\le u(t)}}$	thanks to (G1) and (G2). Thus, we get
	\begin{gather*}
	%\begin{gathered}
	P_{f,u} \le \ex{\ind{\forall t\in\T\ \ X_t+f-\tilde f\le u(t)} \exp \Bigl  \{ - \frac{1}{2} \norm{ \tilde f}_{\mathbb{H}_X}^2 + \Theta(\tilde \gamma, u) \Bigr \}}\\
	= P_{f-\tilde f,u} \exp \Bigl  \{ - \frac{1}{2} \norm{ \tilde f}_{\mathbb{H}_X}^2 + \Theta(\tilde \gamma, u) \Bigr \}
	%\exp\set{-\frac12 \bnorm{\tilde \gamma}_{\mathcal H_X}^2 + \tilde \gamma(T)u(T) +\int_0^T  u(s)d\bigl(-\tilde \gamma(s)\bigr) }.
	%\end{gathered}
	\end{gather*}
	establishing  the claim.
	
	2.  From assumption (G3), namely $\tilde f\ge f$,  we obtain similarly to \eqref{eq:cameron-martin}
	\begin{gather*}
	P_{f,u} = \ex{\ind{\forall t\in\T\ \ X_t+f(t)\le u(t)}} \ge  \ex{\ind{\forall t\in\T\ X_t+\tilde f(t)\le u(t)}}\\ = \ex{\ind{\forall t\in\T\ X_t\le u(t)}\exp \Bigl \{- \frac{1}{2}\bnorm{\tilde f}_{\HK}^2 + \sprod{X,\tilde \gamma}}\\
	\ge \ex{\ind{\forall t\in\T\ u_{-}(t)\le X_t\le u(t)}\exp \Bigl \{- \frac{1}{2}\bnorm{\tilde f}_{\HK}^2 + \sprod{X,\tilde \gamma} }.
	\end{gather*}
	Also, similarly to \eqref{eq:integral-estimate}, we obtain
	$$
	\sprod{X,\tilde \gamma} \ge   \Theta(\tilde \gamma, u_-)
	$$
	on $\set{\forall t\in\T\ u_{-}(t)\le X_t\le u(t)}$, whence
	\begin{gather*}
	P_{f,u} \ge P_{0,u,u_{-}} \exp \Bigl  \{ - \frac{1}{2} \norm{ \tilde f}_{\mathbb{H}_X}^2 + \Theta(\tilde \gamma, u_-) \Bigr \}.
	% \exp\set{-\frac12 \bnorm{\hat g}_{\mathcal H_X}^2 + \hat g(T)u_{-}(T) +\int_0^T  u_{-}(t)d\bigl(-\hat g(t)\bigr)}.
	% \qedhere
	\end{gather*}
\end{proof}
\begin{remark}
	From (G1) and (G3) it follows that $\sprod{f-\tilde f,\tilde \gamma}\le 0$, so (G2) holds as an equality whenever (G1)--(G3) are satisfied simultaneously for some $\tilde \gamma$ and $\tilde f$ (not necessarily equal to $\mathcal R \tilde \gamma$). Moreover, in this case $\tilde \gamma$ and $\tilde f-f$ must be ``orthogonal'' in the sense that $\tilde \gamma$ is supported by the set $\set{t\in\T_1: \tilde f(t) - f(t)=0}$. 
\end{remark}

Now we turn to the question of identification of $\tilde f$ and $\tilde \gamma$ satisfying (G1)--(G3). To this end, for any $f\in \HK$, consider the following minimization problem:
\begin{equation}\label{eq:minprob}
\text{minimize }\norm{h}_{\HK}\text{ for all }h\in \HK, h\ge f;
\end{equation}
here the comparison is understood, as usual, in the pointwise sense, i.e. $h\ge f$ means $h(t)\ge f(t)$ for all $t\in \T$.

\begin{lemma} \label{lemma:closed}
	The set $C_f := \set{h\in \mathbb H_X \mid \forall t\in\T\ \ h(t)\ge f(t) }$ is a closed set in $\mathcal{H}_X$.
\end{lemma}
\begin{proof}
	Since $\mathbb H_X$ consists of continuous functions, we can consider the identity operator $\operatorname{id}_{\mathbb H_X}$ as acting from $\mathbb H_X$ to $C_0(\T;\T_0)$. It is obviously closed, so by the continuous graph theorem it is continuous. Consequently, the set $C_f$, which is closed in $C_0(\T;\T_0)$, is also closed in $\mathbb H_X$. %Therefore, $C_f = \mathcal R^{-1} A_f$ is closed as a preimage of a closed set.
\end{proof}

Since the set $C_f$ is convex and closed in $\mathbb{H}_X$, then by \cite[Chapter 1]{edwards}, there exists a unique element $\tilde f$ solving the  minimization problem \eqref{eq:minprob}. Moreover, the following proposition holds.

\begin{proposition} \label{prop:G2}
	The solution to the minimization problem \eqref{eq:minprob} satisfies
	\begin{equation}\label{eq:G2}
	\bprod{f-\tilde f,\tilde f}_{\HK} = 0.
	\end{equation}
\end{proposition}

\begin{proof}
	Since $\tilde f$ is a metric projection of $0$ to the set $C_f$, %by \cite[Lemma]{CG59}, 
	the solution $\tilde f$ is characterized by the well-known variational inequality (see e.g., \cite[Lemma]{CG59} or \cite[Lemma 2.2]{Noor}) %following property: 
	\begin{equation}\label{eq:projection}
	\bnorm{\tilde f}^2_{\HK}\le  \bprod{\tilde f,h}_{\HK}\quad  \forall h\in C_f.
	\end{equation}
	Plugging $h=f$ to \eqref{eq:projection}, we get
	$\bprod{f-\tilde f,\tilde f}_{\HK}\ge 0$; plugging $h=2\tilde f -f$, we obtain $\bprod{f-\tilde f,\tilde f}_{\HK}\le 0$, establishing the claim.
\end{proof}

In other words, any $\tilde \gamma$ such that $\mathcal R \tilde \gamma = \tilde f$ satisfies (G2) and (G3). To ensure the validity of (G1), we need an additional assumption. %Denote $\mathbb H_X^+ = \set{f\in \mathbb H_X\mid f\ge 0}$.

\begin{itemize}
	\item[(P)] If $g\in \mathcal H_X$ is such that $\sprod{f,g} \ge 0$ for all $f\in \mathbb H^+_X:=\set{f\in \mathbb H_X\mid f\ge 0}$, then $g\in M^+(\T_1)$.
\end{itemize}

\begin{proposition}\label{prop:G1-G3}
	Under the assumption (P), the solution $\tilde f=\mathcal R \tilde \gamma$ to the minimization problem \eqref{eq:minprob} satisfies (G1)--(G3).
\end{proposition}

\begin{proof}
	If $k\in \mathbb{H}^+_X$,  then $\tilde f+k\in C_f$,  hence  by \eqref{eq:projection}
	$$
	\bnorm{\tilde f}^2_{\HK}\le 	 \bprod{\tilde f,\tilde f+k}_{\HK} =  	\bnorm{\tilde f}^2_{\HK}+ \bprod{\tilde f,k}_{\HK},
	$$
	whence
	$$
	\sprod{k,\tilde \gamma} = \bprod{k,\tilde f}_{\HK} \ge 0.
	$$
	Since $k\in \mathbb{H}^+_X$ is arbitrary, then $\tilde \gamma\in M^+(\T_1)$ by assumption (P), i.e.\ we have (G1). (G2) follows from Proposition~\ref{prop:G2}, whereas  (G3) follows from the definition of $\tilde f$.
\end{proof}

\subsection{\CE{Asymptotics}}% for boundary non-crossing probabilities}
%rev2

In this section we derive expansions for $\log P_{\cY f, u}$ as  $\cY$ tends to infinity.  
We will need the following additional assumption.
\begin{itemize}
	\item[(D)] $\mathcal R$ is injective on $M(\T_1)$.
\end{itemize}
\begin{remark}
	Condition (D) is equivalent to the distribution of $X$ having full support, i.e.\  the support of the distribution of $X$ coincides with $C_0(\T;\T_0)$. Indeed, it is well known (see e.g.\ \cite[Lemma 5.1]{VZ08}) that the support of distribution of $X$ is the closure of $\mathcal R M(\T_1)$. If the latter were not $C_0(\T;\T_0)$, then by Hanh--Banach theorem there would exist non-zero $\gamma\in M(\T_1)$ such that $\sprod{f,\gamma} = 0$ for all $f\in \mathcal R M(\T_1)$. In particular, $\sprod{\mathcal R\gamma,\gamma}=0$, which would contradict the injectivity. On the other hand, if $\mathcal R\gamma = 0$ for some non-zero $\gamma$, then $\sprod{f,\gamma} = 0$ for all $f\in \mathcal R M(\T_1)$, hence, for all $f$ from the support of $X$, which then cannot be full. 
	
	We have chosen the injectivity assumption because we believe it is easier to verify than the full support property.
\end{remark}

Now we state the assumptions on the boundary function $u$. 

\begin{itemize} \item[(U)]  There exists a sequence $(u_n,n\ge 1)$ of continuous functions such that 
	\begin{itemize}
		\item[\rm 1)] $u_n(t)\uparrow u(t)$, $n\to \CE{+}\infty$, for all $t\in\T_1$;
		\item[\rm 2)] $P_{0,u,u_n} = \pr\big(\forall t\in\T \  u_{n}(t)\le X_t\le u(t)\big)>0$ for all $n\ge 1$.
	\end{itemize}
\end{itemize}
\begin{remark}\label{prop:U}
	Under assumption (D), a sufficient condition for a lower semicontinuous $u\colon \T\to \R$ to satisfy (U) is that  $u(t)>0$ for all $t\in\T_0$.
	Indeed, in this case for any $u_-\in C(\T)$ such that $u_-(t)<0$ for all $t\in \T_0$ and $u_-(t)<u(t)$ for all $t\in\T$, the set $$
	A_{u,u_-} = \set{g\in C_0(\T;\T_0) \mid \forall t\in\T\ \ u_-(t)<g(t)< u(t)}
	$$
	is non-empty and open in $C_0(\T;\T_0)$. Therefore, since the support of distribution of $X$ is $C_0(\T;\T_0)$, then we have 
	$$P_{0,u,u_-} = \pr(X \in A_{u,u_-})>0.$$
	Consequently, (U) holds for any sequence of continuous functions $u_n\in C(\T)$ such that $u_n(t)<0$, $t\in T_0$,  and $u_n(t)\uparrow u(t)$, $n\to\CE{+}\infty$, for any $t\in \T_1$.
	
	We believe that (U) holds whenever $u(t)\ge 0$ for $t\in \T_0$ and $P_{0,u}>0$. However, the above argument fails, as the set $A_{u,u_-}$ can be empty. Considering the set $$B_{u,u_-}:=\set{g\in C_0(\T;\T_0) \mid \forall t\in\T_1\ \ u_-(t)\le g(t)\le u(t)}$$ will not help, as it is not open in general. One could consider a finer topology to overcome this problem, but then the dual space would be larger and perhaps not as tractable as $M(\T_1)$. 
\end{remark}

\begin{theorem}\label{thm:assym-gen}
	Assume that (D) holds, $f \in \mathbb H_X$ and let $u\colon \T\to\R$ be a lower semicontinuous function satisfying (U).
	If there exists $\tilde f = \mathcal R\tilde \gamma\in \mathbb H_X$ satisfying  (G1)--(G3), then
	\begin{equation}\label{eq:P_cf-asymp}
	\begin{aligned}
	\log P_{\cY f,u} &= - \frac{\cY^2}{2}\norm{ \CE{\tilde \gamma}}^2_{\mathcal H_X} + \cY\, \Theta(\CE{\tilde \gamma}, u) +o(\cY)\\ &\CE{= - \frac{\cY^2}{2}\norm{ \tilde f}^2_{\HK} + \cY\, \Theta(\tilde \gamma, u) +o(\cY)},
	\quad \cY\to+\infty.
	\end{aligned}
	\end{equation}
\end{theorem}
\begin{remark}  \label{remXh}
	It follows from \eqref{eq:P_cf-asymp} that all $\tilde \gamma\in \mathcal H_X$ satisfying (G1)--(G3) must have equal norms. Therefore, since the set of such functions is convex, they all must coincide in $\mathcal H_X$ implying that such $\tilde\gamma\in \mathcal H_X$ is unique. %Below we give a sufficient condition under which such $\tilde\gamma$ exists, and in this case G2 holds as an equality. We suspect that this is always the case, i.e.\ that G1--G3 imply that there is an equality in G2, but we were not able to show this.
\end{remark}
\begin{proof}
	Denote the right-hand side of \eqref{eq:P_cf-asymp} by $r(\cY,\tilde \gamma,u)$. Since $P_{0,u}>0$, the inequality
	\eqref{eq:noncrossingprobupperbound-verygeneral} yields 
	$$
	\limsup_{\cY\to+\infty} \big(\log P_{\cY f,u} - r(\cY,\tilde \gamma,u)\big) \le \limsup_{\cY\to+\infty} \log P_{\cY(f-\mathcal R \tilde \gamma),u} \le 0,
	$$
	so it remains to establish the lower bound.
	
	Next, take the sequence $(u_n,n\ge 1)$ satisfying (U). It is clear that one can choose positive integers \CE{$(k=k(\cY),\cY\ge 0)$} growing to $\CE{+}\infty$ as $\cY \to +\infty$ sufficiently slowly so that $\cY^{-1}\log P_{0,u,u_{k}}\to 0$, $\cY\to+\infty$. Then for any $n\ge 1$, by  \eqref{eq:lower} we get
	$$
	\liminf_{\cY\to+\infty} \cY^{-1}\big(\log P_{\cY f,u} - r(\cY,\tilde \gamma,u_{k})\big) \ge \liminf_{\cY\to+\infty} \cY^{-1}\log P_{0,u,u_{k}} =0.
	$$
	Consequently,
	\begin{gather*}
	\liminf_{\cY\to+\infty} \cY^{-1}\big(\log P_{\cY f,u} - r(\cY,\tilde\gamma,u)\big)\\
	\ge \liminf_{\cY\to+\infty} \cY^{-1}\big(\log P_{\cY f,u} - r(\cY,\tilde\gamma,u_{k})\big) + \liminf_{\cY\to+\infty} \cY^{-1}\big(r(\cY,\tilde\gamma,u_{k}) - r(\cY,\tilde \gamma,u)\big) \\
	\ge  \liminf_{\cY\to+\infty} \bigl(\Theta(\tilde \gamma,u_{k})-\Theta(\tilde \gamma,u)\bigr)= \liminf_{\cY\to+\infty}  \int_{\T_1}  \big(u_{k}(t)-u(t)\big)\tilde \gamma(dt).
	\end{gather*}
	Thanks to the dominated convergence $\lim_{\cY \to + \infty} 
	\int_{\T_1}  \big(u_{k}(t)-u(t)\big)\tilde\gamma(dt)= 0$ implying 
	$$
	\liminf_{\cY\to+\infty} \cY^{-1}\big(\log P_{\cY f,u} - r(\cY,\tilde\gamma,u)\big) \ge  0,
	$$
	hence the proof is complete. 
\end{proof}
%Now we investigate when the assumptions of Theorem~\ref{thm:assym-gen} are fulfilled.
%
%
%
%
%
%
%
%\begin{corollary} If $g \in \HKX$ satisfies condition G1 and $X$ is a continuous process,  then $g$
%is the projection of $0$ to the set $C_f = \set{h\in \mathcal H_X\mid \forall t\in[0,T]\   \mathcal R h(t)\ge f(t)}$, i.e.\ $\norm{ g}_{\mathcal H_X} = \min_{h\in C_f} \norm{h}_{\mathcal H_X}$.	
%\end{corollary}

We are now ready to state the main result of this section.
\begin{theorem}\label{thm:main}
	Let  (D) and (P) hold, $ f\in\mathbb H_X$  and let $u$ be a lower semicontinuous function satisfying (U). If further $\tilde\gamma$ is the projection of $0$ to the set $C_f$, then \eqref{eq:P_cf-asymp} holds.
\end{theorem}
\begin{proof}
	The statement follows from Proposition~\ref{prop:G1-G3} and Theorem~\ref{thm:assym-gen}.
\end{proof}

%\begin{remark}
%Thanks to \eqref{eq:pairing},  $$
%\pprod{ g-\tilde\gamma,\tilde\gamma}_{\mathcal H_X} = \sprod{\mathcal{R}( g-\tilde\gamma),\tilde\gamma} = \sprod{f - \mathcal{R}\tilde\gamma,\tilde\gamma},
%$$
%and this is zero by Proposition~\ref{prop:G2}. However, since $f\le \tilde f:=\mathcal R \tilde \gamma$, and $\tilde \gamma$ is non-negative, this is only possible if
%$$
%\int_0^T \big(f(t)- \tilde f(t)\big)\tilde \gamma(dt) =0.
%$$
%The latter implies that $\tilde \gamma(t) = 0$ whenever $f(t)<\tilde f(t)$. In other words, $\tilde \gamma(t) = \tilde \gamma((t,T])$ satisfies (in the weak sense) the variational inequality
%$$
%\max\big(f(t) - \tilde f(t),\tilde \gamma'(t)\big)  = 0.
%$$
%\end{remark}
In general, it is difficult to identify $\tilde \gamma$. But there are cases where it is possible, e.g.\ if the drift is the covariance operator applied to a non-negative measure. Namely, the following result follows from Theorem~\ref{thm:main}  immediately.
\begin{corollary}\label{cor:positive}
	Assume that  (D) holds and $u$ is a lower semicontinuous function satisfying (U). Then for any $\gamma\in M^+(\T_1)$ and $f=\mathcal R \gamma$  the asymptotic expansion \eqref{eq:P_cf-asymp} holds with $\tilde \gamma=\gamma$.
\end{corollary}

\subsection{Locally compact parameter space}\label{sec:localcomp}

Let now the continuous centered Gaussian process $X$ be indexed by a separable metric space $\T$, which we will assume here to be non-compact, but locally compact. We want to reduce our problem to its counterpart with compact (separable) parameter set. To this end, denote by $\overline \T = \T \cup \{t_\infty\}$ the one-point compactification of $\T$. We first show that $X$ can be multiplied by some positive function so that the product vanishes at infinity. 
\begin{lemma}\label{lem:vanisher}
	There exists a continuous  function $v\colon \T \to (0,\infty)$ such that $v(t)X_t\to 0$, $t\to t_\infty$  a.s.
\end{lemma}
\begin{proof}
	Being a separable metric space, the space $\T$ is Lindel\"of, so in view of local compactness there exists a countable family $\set{\T_n, n\ge 1}$ of  compact sets such that $\T = \bigcup_{n\ge 1} \T_n$ and $\T_n$ is contained in $\T_{n+1}^\circ$, the interior of $\T_{n+1}$, for each $n\ge 1$ (see e.g.\ \cite[Chapter XI, Theorem 7.2]{dugundji}). 
	
	Since $X$ is continuous and for each $n\ge 1$, $\T_n$ is compact, then clearly 
	$$\pr(\sup_{t\in \T_n} |X_t|<\infty)=1.$$
	Therefore, there exists some $a_n>0$ such that $\pr\left(\sup_{t\in \T_n} |X_t|>a_n\right) <2^{-n}$. Without loss of generality, we can assume that $a_n<a_{n+1}$ for each $n\ge 1$ and $a_n\to \infty$, $n\to\CE{+}\infty$. 
	
	For any $n\ge 1$, denote $b_n = a_{n+1}^{-2}$ and $D_n = \partial \T_n := \T_n \setminus \T_n^\circ$. Since $\T_n \subset \T_{n+1}^\circ$, we have $\partial \T_n\cap \partial \T_{n-1} = \varnothing$. Then by Urysohn's lemma, there exists a continuous function $v_n\colon \T_{n+1}\setminus \T_n^\circ\to [b_{n+1},b_n]$ such that $v_n(t) = b_n$, $t\in D_n$, $v_n(t) = b_{n+1}$, $t\in D_{n+1}$. Now set 
	$$
	v(t) = a_1 \ind{\T_1}(t) + \sum_{n= 1}^\infty v_n(t) \ind{\T_{n+1}\setminus \T_n }(t), t\in \mathbb{T}. 
	$$
	By construction, this is a continuous function with $\sup_{\T \setminus\T_n} v(t)\le b_n$, $n\ge 1$. 
	
	On the other hand, by the Borel--Cantelli lemma, with probability $1$ there exists $n_0(\omega)$ such that $\sup_{t\in \T_n} |X_t|\le a_n$, $n\ge n_0(\omega)$. Therefore, for $n\ge n_0(\omega)$
	\begin{gather*}
	\sup_{t\in \T_{n+1}\setminus\T_n} \bigl|v(t) X_t\bigr| \le \sup_{\T \setminus\T_n} v(t)\cdot \sup_{t\in \T_{n+1}} |X_t| \le a_{n+1}^{-2}\cdot a_{n+1} = a_{n+1}^{-1}.
	\end{gather*}
	Consequently, for all $n\ge n_0(\omega)$  
	$$\sup_{t\in \T\setminus\T_n} \bigl|v(t) X_t\bigr| \le a_{n+1}^{-1}\to 0$$
	as $n\to\CE{+}\infty$ establishing the proof. 
\end{proof}
\begin{remark}
	The above lemma is valid for any continuous process on $\T$, since the Gaussian distribution of $X$ is not used in the proof. 
\end{remark}

Now we are ready to state the main result about reduction to the case of compact parameter space.

Namely, set below 
$$\overline X_t = v(t)\CE{X_t},  \quad \bar f(t) = v(t)f(t),  \quad \bar u(t) = v(t)u(t)$$
and putting   
$$\overline{X}_{t_\infty} = \bar f(t_\infty)  =  0,  \quad \bar u(t_\infty) = (\liminf_{t\to t_\infty} v(t)u(t))\wedge 0$$
we have the following statement. 
\begin{theorem}
	The process $\overline X$ is a continuous centered Gaussian process on $\overline \T$ and for any $f\in \mathbb{H}_X$ and any lower semicontinuous $u\colon \T \to \R$,  we have that $\bar f \in \mathbb{H}_{\overline X}$, $\bar u$ is lower semicontinuous and further 
	\begin{equation}\label{eq:loccompprobequality}
	\pr \left(\forall t\in \T \ \ X_t +  f(t)\le  u(t)\right) = \pr \left(\forall t\in \overline {\T}\ \ \overline  X_t + \bar f(t)\le \bar u(t)\right).
	\end{equation}
\end{theorem}
\begin{remark}
	It is important e.g.\ for asymptotic results like \eqref{ldp}  that $\overline X$, $\bar f$, and $\bar u$ depend linearly (and in a rather simple way) on $X$, $f$ and $u$, respectively.
\end{remark}
\begin{proof}
	The process $\overline X$ is obviously centered Gaussian, and Lemma~\ref{lem:vanisher} immediately implies that $\overline X$ is continuous. The fact that $\bar f \in \mathbb H_{\overline{X}}$ is a consequence of the following well-known characterization of the Cameron--Martin space. Namely, it consists of functions $f$ such that the distribution of $X+f$ is absolutely continuous w.r.t.\ that  of $X$. That said, for any $\overline A\subset C(\overline \T)$ such that $\pr(\overline X\in \overline A)=0$,  define $A = \set{h|_{\T}/v, h\in \overline A}$ and write
	$$
	\pr \left(X \in A \right) = \pr \left( \overline  X  \in \overline A\right) = 0.
	$$
	Hence, since  $f\in\mathbb H_X$ we have
	$$
	\pr \left( \overline  X + \bar f \in \overline A\right) = \pr \left(X + f \in A \right) = 0,
	$$
	whence we derive that $\bar f\in \mathbb H_{\overline{X}}$. Equation \eqref{eq:loccompprobequality} is obtained similarly. 
	
	It remains to remark that $\bar u$ is lower semicontinuous by definition. (It is possible that $\bar u(t_\infty)=-\infty$, but in this case, and more generally in the case where $u(t_\infty)<0$  both probabilities in question are equal to zero.)
\end{proof}

\CE{\subsection{Relation to the large deviation principle}\label{sec:ldp}
	
	The asymptotics \eqref{ldp} is also closely related to the large deviation principle. Since for process $X$ is a centered Gaussian element in the separable Banach space $C_0(\T;\T_0)$, we can apply the general large deviation principle  by Donsker and Varadhan \cite{DV76} (see also \cite[Section 3.4]{DS89}, \cite[Theorem 4.5]{DGL96}): for any Borel set $A\subset C_0(\T;\T_0)$, 
	\begin{equation}\label{eq:gen-ldp}
	- \inf_{A^\circ} I(x) \le \liminf_{\eps \to 0+} \eps^2 \log \pr (X\in A) \le  \limsup_{\eps \to 0+} \eps^2 \log \pr (X\in A) \le - \inf_{\overline A} I(x),
	\end{equation}
	where $A^\circ$ and $\overline{A}$ are the interior and the closure of $A$, respectively. If we assume, as before, the injectivity of the covariance operator, the rate functional can be identified through the concept of Wiener quadruple (see \cite[p. 88]{DS89}): the Banach space $C_0(\T;\T_0)$ together with the Hilbert space $\HK$, the identity map $S\colon \HK\to C_0(\T;\T_0)$ (which is injective thanks to our assumption) and the distribution of $X$ forms a Wiener quadruple, so the rate functional is given by \cite[Theorem 3.4.12]{DS89}:
	$$
	I(x) = \begin{cases}
	\frac 12 \norm{x}^2_{\HK}, & x\in \HK,\\
	\infty, & x\notin \HK.
	\end{cases}
	$$

	To relate the large deviation extimates \eqref{eq:gen-ldp} to the boundary non-crossing probability $P_{\cY f,u}$, denote $\eps = \cY^{-1}$, $$A_{ f} = \set{g\colon \T\to \R\mid \forall t\in \T\ g(t)\le -f(t)}.$$ 
	Then the  boundary non-crossing probability can be written as $P_{\cY f,u}  = \pr\bigl(\eps X \in A_{f-\eps u} \bigr)$, however, \eqref{eq:gen-ldp} is not directly applicable since the target set $A_{f-\eps u}$ depends on $\eps$. To overcome this problem, one may fix some $\eps_0>0$ and write 
	\begin{equation*}%\label{eq:false-ldp}
	\begin{gathered}
	\limsup_{\eps \to 0+} \eps^2  \log \pr\bigl(\eps X \in A_{f - \eps u} \bigr) \le 
	\limsup_{\eps \to 0+} \eps^2  \log \pr\bigl(\eps X \in A_{f -\eps_0 u^+} \bigr)\\
	\le  - \inf_{\overline A_{f - \eps_0 u^+}} I(x) =  - \inf_{A_{f -\eps_0 u^+}} I(x),
	\end{gathered}
	\end{equation*}
	where $u^+(t) = \max\{u(t),0\}$. Then, letting $\eps_0\to \CE{+}0$, we get 
	$$
	\limsup_{\eps \to 0+} \eps^2  \log \pr\bigl(\eps X \in A_{f - \eps u} \bigr) \le 
	- \inf_{A_{f}} I(x).
	$$
	But it is clear that 
	$$\inf_{\overline A_{f}}I(x) = \frac12 \inf\set{\norm{g}^2_{\HK}: g\le -f} = 
	\frac12 \inf\set{\norm{h}^2_{\HK}: h\ge f} = \frac{1}{2}\bnorm{\tilde{f}}^2_{\HK},
	$$
	where $\tilde f$ is, as before, the solution to the constrained minimization problem \eqref{eq:minprob}. As a result, going back to our notation,
	\begin{equation}\label{LD:upper}
	\lim_{\cY\to +\infty} \cY^{-2}\log P_{\cY f,u} \le  -\frac{1} 2 {\bnorm{\tilde f}_{\mathbb{H}_X}^2 }.
	\end{equation}
	However, in the case where $\T_0\neq \varnothing$, there is no clear way how to get a lower bound from \eqref{eq:gen-ldp}: when $u$ is non-negative, the same approach gives 
	$$
	\limsup_{\eps \to 0+} \eps^2  \log \pr\bigl(\eps X \in A_{f-\eps u} \bigr) \ge 
	\limsup_{\eps \to 0+} \eps^2  \log \pr\bigl(\eps X \in A_{f} \bigr)\\
	\ge  - \inf_{A^\circ_{f}} I(x).
	$$
	Since both $X$ and $f$ vanish on $\T_0$, $A_f$ has empty interior, so $\inf_{A^\circ_{f}} I(x) = \infty$; the lower estimate given by the large deviation principle  sharp, as $\pr\bigl(\eps X \in A_{f} \bigr) = 0$ in many cases, e.g., for Brownian motion on $[0,T]$. 
	
	As it was mentioned in Introduction, there is another (simpler) way to derive a lower bound, leading to \eqref{lower:LD}, which, combined with \eqref{LD:upper}, yields
	$$
	\log P_{\cY f,u}  \sim  -\frac{\cY ^2} 2 {\bnorm{\tilde f}_{\mathbb{H}_X}^2 }, \quad \cY\to +\infty.
	$$ 
	Unfortunately, this gives only the main term of asymptotic expansion \eqref{ldp}. The next term of the asymptotics comes from the following heuristics: the target set $A_{f - \eps u}$ is almost $A_f$ with slightly perturbed boundary. Therefore, denoting by $\Lambda(A)$ the ``rate functional'' corresponding to the set $A$, and assuming some smoothness, one might expect that for $\eps\to 0+$,
	\begin{equation}\label{eq:false-ldp}
	\Lambda(A_{f-\eps u}) \approx \Lambda(A_f) - \eps \sprod{u,\Lambda'_f(A_f)} = \frac{1}{2}\bnorm{\tilde f}_{\mathbb{H}_X}^2 - \eps \sprod{u,\Lambda'(A_f)},
	\end{equation}
	where $\Lambda'_f$ is the derivative in some sense of $\Lambda(A_f)$ with respect to $f$, and this relation looks similar to \eqref{ldp}. In certain situations, this heuristic argument may be given a precise meaning: see, for example, \cite[Theorem 9.3.2]{Bor13}. However, in our case such argument would most likely fail. Indeed, if it were possible to validate, it would also work for negative ``perturbations''. However, if $\T_0\neq \varnothing$ and $u$ is positive, then $\pr(X \in A_{f+\eps u}) = 0$, since $X$ and $f$ both vanish on $\T_0$, so \eqref{eq:false-ldp} cannot provide correct logarithmic asymptotics.

	On the other hand, the large deviation estimates may be used to derive the asymptotic behavior of the probabilities $\pr(\eps X\in A_{f-\delta u})$ when $\eps\to 0+$ and then $\delta\to 0+$, similarly to the results for random walks, established in \cite{Bor67}, but such questions are beyond the scope of our article. 
}

\section{Applications}\label{sec:examples}
In this section we specialize the general results of Section~\ref{sec:general} to several one-parameter processes. In all the examples, we skip the routine verification of assumptions (P) and (D) while putting more emphasis on relevant details.

Throughout the section,  $W = \{W_t,t\in \mathbb{R}\}$ is a standard Wiener process on $(\Omega,\F,\pr)$. By $AC([a,b])$ we will denote the set of absolutely continuous functions defined on $[a,b]$, and for $f\in AC([a,b])$, $f'$ will denote its weak derivative.

\subsection{Wiener process on $[0,b]$}\label{ex:wiener[0,b]}
Let $\T = [0,b]$ and $X = W$. Now $\T_0=\{0\}$, $\T_1 = \T\setminus \T_0 = (0,b]$. The primary space is $C_0([0,b]) = \set{x\in C([0,b]): x(0) = 0}$ with dual $M((0,b])$, the space of finite signed measures on $(0,b]$. The covariance operator  is given by 
\begin{equation}\label{eq:cov-wiener}
\begin{gathered}
\mathcal R \mu(t) = \int_0^T \min (t,s) \mu(ds) =  \int_0^T \int_0^t \ind{[0,s]}(u)du\, \mu(ds) 
\\= \int_0^t \int_0^T \ind{[0,s]}(u)\mu(ds) du
= \int_0^t \mu([u,b])du = \int_0^t J\mu (u)du,
\end{gathered}
\end{equation}
where $J\mu(u) = \mu([u,b])$, $u\in[0,b)$, $J\mu(T) = \mu(\set{T})$. Similarly, for $\mu,\nu\in M((0,b])$ 
\begin{gather*}
\sprod{\mathcal R \mu,\nu} = \int_0^b\int_0^b \min (t,s) \mu(ds)\nu(dt) = \int_0^b \int_0^b\int_0^b\ind{[0,t]}(u) \ind{[0,s]}(u)du\, \mu(ds)\nu(dt) \\
= \int_0^b \int_0^b \ind{[0,s]}(u)\mu(ds) \int_0^b\ind{[0,b]}(u)\nu(dt) du = \int_0^b J\mu(u)J\nu(u)du = \big(J\mu,J\nu\big)_{L^2[0,b]}.
\end{gather*}
Consequently,  $J$ extends to an isomorphism between $\mathcal H_X$ and $L^2[0,b]$; the image is full since $J\mu$ can be arbitrary left-continuous bounded variation function. Therefore, in view of \eqref{eq:cov-wiener}, the image of $\mathcal H_X$ under the covariance operator consists of functions of the form $\int_0^t h(u)du$, where $h\in L^2[0,b]$, which is the well-known description of the Cameron-Martin space of $W$. It is worth mentioning that for $\mu\in M((0,b])$ and $f = \mathcal R\mu$  we have 
\begin{gather*}
\int_0^t f'(t)dW_t = \int_0^b J\mu(t) dW_t = \int_0^b \mu([t,b])dW_t = \int_0^b \int_t^b \mu(ds) dW_t\\
= \int_0^b \int_0^s dW_t\, \mu(ds)  = \int_0^b W_s\, \mu(ds) = \sprod{W,\mu},
\end{gather*}
so the Cameron-Martin density can be transformed to its more familiar form:
\begin{gather*}
\mathcal{E}_W(\mu) = \exp\Bigl\{\sprod{W,\mu} - \frac12\norm{\mu}^2_{\mathcal H_X} \Bigr\} = \exp\Bigl\{\int_0^t f'(t)dW_t - \frac12\norm{J\mu}^2_{L^2[0,b]} \Bigr\}\\ = \exp\Bigl\{\int_0^t f'(t)dW_t - \frac12\norm{f'}^2_{L^2[0,b]} \Bigr\}.
\end{gather*}

Further, the image of a non-negative finite measure on $(0,b]$ is an  absolutely continuous function $f$ with $f(0) = 0$ and with a non-increasing non-negative derivative. Equivalently, this is a  concave  non-decreasing function with $f(0) = 0$. Therefore, in order to identify the function $\tilde f$ from Theorem~\ref{thm:assym-gen}, which corresponds to the drift $f = \mathcal R \gamma$, we need to find a concave non-decreasing function $\tilde f\ge f$ such that $\tilde \gamma = \mathcal R^{-1}\tilde f$ satisfies (G2). The latter is equivalent to 
$$
\pprod{J \gamma - J\tilde \gamma,J\tilde \gamma}_{L^2[0,b]}\ge 0,
$$
which, in view of \eqref{eq:cov-wiener}, reads
$$
\big(f'-\tilde f',\tilde f'\big)_{L^2[0,b]}\ge 0. 
$$
Thanks to Theorem~\ref{thm:concmaj}, this property (even with equality) is satisfied by the least non-decreasing concave majorant of $f$, which also is a solution to the minimization problem~\eqref{eq:minprob}. This is not surprising, as we recover the well-known results for the Wiener process (see e.g. \cite{BHH07}), which we summarize below. Note also that, by definition of $J\gamma$, we should define $\tilde f'$ to be left-continuous on $(0,b]$ and continuous at zero, so we should take the left derivative for $t\in(0,b]$ and the right derivative at $0$.
\begin{theorem}\label{thm:wiener[0,T]}
	Let $u\colon [0,b]\to \R$ be lower semicontinuous,  $f\in AC([0,b])$ be such that $f(0)=0$ and $f'\in L^2[0,b]$, $\tilde f$ be the least non-decreasing concave majorant of $f$, and $\tilde f'_-(t)$ be its left derivative (right derivative for $t=0$).
	
	1. The probability $P_{f,u} = \pr(\forall t\in[0,b]\ W_t+f(t)\le u(t))$ admits the upper bound
	\begin{equation*} 
	P_{f,u} \le P_{f-\tilde f,u} \exp \Bigl \{- \frac{1}{2}\bnorm{\tilde f'_-}_{L^2[0,b]}^2 + \int_0^b u(t) d\bigl(-\tilde f'_-(t)\bigr)\Bigr \}.
	\end{equation*}
	
	2. For any $u_{-}\in C([0,b])$ such that $u_{-}(t) < u(t)$ for all $t\in [0,b]$ and 
	$$
	P_{0,u,u_-} := \pr\big(\forall t\in[0,b]\ \ u_{-}(t)\le W_t\le u(t)\big)>0,
	$$
	the probability $P_{f,u}$ admits the lower bound 
	\begin{equation*}
	\begin{gathered}
	P_{f,u} \ge P_{0,u,u_{-}}  \exp \Bigl  \{ - \frac{1}{2} \bnorm{ \tilde f'_-}_{L^2[0,b]}^2 + \int_0^b u_-(t) d\bigl(-\tilde f'_-(t)\bigr) \Bigr \}.
	\end{gathered}
	\end{equation*}
	
	3. If $u(0)>0$, then the following asymptotics holds:
	$$
	\CE{\log}\, P_{\cY f,u} = - \frac{\cY ^2}{2} \bnorm{ \tilde f'_-}_{L^2[0,b]}^2 + \cY \int_0^b u(t) d\bigl(-\tilde f'_-(t)\bigr) +o(\cY ), \quad \cY \to+\infty. 
	$$
	
\end{theorem}

\subsection{Wiener process on $[a,b]$}

Let again  $X = W$, but $\T = [a,b]$ with $a<b$, $ab\neq 0$. There are two different cases depending on whether $0\in [a,b]$ or not. 

\subsubsection{$0\notin [a,b]$}
In this case $\T_0=\varnothing$, $\T_1 = [a,b]$. The primary space is $C([a,b])$ with dual $M([a,b])$, the space of finite signed measures on $[a,b]$. Without loss of generality, we can assume $a>0$.  

Similarly to \eqref{eq:cov-wiener}, the covariance operator  is given by 
\begin{equation}\label{eq:cov-wiener-ab}
\begin{gathered}
\mathcal R \mu(t) = \int_a^b \int_0^{t} \ind{[0,s]}(u)du\, \mu(ds) 
= \int_0^t \int_{a}^b \ind{[0,s]}(u)\mu(ds) du\\
= \int_0^t \mu([u\vee a,b])du = \int_0^t J_a\mu (u)du, t\in [a,b],
\end{gathered}
\end{equation}
where $J_a\mu(u) = \mu([u\vee a,b])$; also for $\mu,\nu\in M([a,b])$ we have  
\begin{gather*}
\sprod{\mathcal R \mu,\nu} = 
\int_a^b J_a\mu(u)J_a\nu(u)du = \big(J_a\mu,J_a\nu\big)_{L^2[0,b]}.
\end{gather*}
As above, the operator  $J_a$ extends to an isomorphism between $\mathcal H_X$ and some subspace of $L^2[0,b]$. The image is now not full, since, for each $\mu\in M([a,b])$, $J_a \mu$ is constant on $[0,a]$; in fact, it is easy to see that $J_a \mathcal H_X$ consists of square integrable functions which are constant on $[0,a]$. Then, by \eqref{eq:cov-wiener-ab}, the image of $\mathcal H_X$ under the covariance operator consists of absolutely continuous functions on $[a,b]$ with square integrable derivative. The image of $M^+([a,b])$ is a bit trickier. As in the previous example, by \eqref{eq:cov-wiener-ab}, it contains concave non-decreasing functions, but not all of them. In fact, it is easy to see from \eqref{eq:cov-wiener-ab} that we must have $f'_{+}(a) = \mu((a,b])\le \mu([a,b])= f(a)/a\ge 0$; also every concave non-decreasing function with such property belongs to the image. Now the function $\tilde f$ from Theorem~\ref{thm:assym-gen} corresponding to  the drift $f = \mathcal R \gamma$ is a concave non-decreasing function such that $\tilde f'_+(a) \ge \tilde f_a(a)\ge 0$ and  $\tilde \gamma = \mathcal R^{-1}\tilde f$ satisfies (G2). As in the previous example, it is possible to identify this function. Namely, thanks to the isomorphism property of $J_a$, we can rewrite (G2) as
\begin{equation}\label{eq:pprodja-j}
\pprod{J_a \gamma - J_a\tilde \gamma,J_a\tilde \gamma}_{L^2[0,b]}\ge 0.
\end{equation}
We have $J_a \gamma(t) = f'(t)$ for $t\in [a,b]$, and $J_a g$ is constant on $[0,a]$ with $\int_0^a J_a \gamma(t) = f(a)$. So, if we extend $f$ to $[0,a]$ linearly, i.e.\ $f(t) = t f(a)/a$, $t\in[0,a]$, then we have $J_a \gamma(t) = f'(t)$ for $t\in [0,b]$. Then, for the least concave non-decreasing  majorant $\tilde f$ of $f$, we have by Theorem~\ref{thm:concmaj} that
$$
\big(f'-\tilde f',\tilde f'\big)_{L^2[0,b]}= 0. 
$$
Moreover, $\tilde f'$ is clearly constant on $[0,a]$, so we have $\tilde f'_- = J_a \tilde \gamma$, where $\tilde \gamma$ is the non-negative measure on $[a,b]$ given by $\tilde \gamma([t,b]) = -\tilde f'_-(t)$, $t\in [a,b]$, and \eqref{eq:pprodja-j} follows. Hence we arrive at the following result.

\begin{theorem}\label{thm:wiener[a,b]}
	Let $u\colon [a,b]\to \R$ be lower semicontinuous and $f\in AC([a,b])$  be such that $f'\in L^2[a,b]$. Define $f(t) = t f(a)/a$ for $t\in [0,a]$, let $\tilde f\colon [0,b]\to \R$ be the least non-decreasing concave majorant of $f$ on $[0,b]$ and $\tilde f'_-$ be its left derivative.
	
	1. The probability $P_{f,u} = \pr(\forall t\in[a,b]\ W_t+f(t)\le u(t))$ admits the upper bound
	\begin{equation*} 
	P_{f,u} \le P_{f-\tilde f,u} \exp \Bigl \{- \frac{1}{2}\bnorm{\tilde f'_-}_{L^2[0,b]}^2 + \int_a^b u(t) d\bigl(-\tilde f'_-(t)\bigr)\Bigr \}.
	\end{equation*}
	
	2. For any $u_{-}\in C([a,b])$ such that $u_{-}(t) < u(t)$ for all $t\in [a,b]$ and
	$$
	P_{0,u,u_-} := \pr\big(\forall t\in[a,b]\ \ u_{-}(t)\le W_t\le u(t)\big)>0,
	$$
	the probability $P_{f,u}$ admits the lower bound 
	\begin{equation*}
	\begin{gathered}
	P_{f,u} \ge P_{0,u,u_{-}}  \exp \Bigl  \{ - \frac{1}{2} \bnorm{ \tilde f'_-}_{L^2[0,b]}^2 + \int_a^b u_-(t) d\bigl(-\tilde f'_-(t)\bigr) \Bigr \}.
	\end{gathered}
	\end{equation*}
	
	3. The following asymptotics holds:
	$$
	\CE{\log}\, P_{\cY f,u}= - \frac{\cY ^2}{2} \bnorm{ \tilde f'_-}_{L^2[0,b]}^2 + \cY \int_a^b u(t) d\bigl(-\tilde f'_-(t)\bigr) + o(\cY ), \quad \cY \to+\infty. 
	$$
	
\end{theorem}
\begin{remark}
	Actually, this example can be compared with the previous one. Namely, we can informally write 
	$$
	P_{f,u} = \pr\big(\forall t\in[a,b]\ \ W_t + f(t)\le u(t)\big) \approx 
	\pr\big(\forall t\in[0,b]\ \ W_t + f(t)\le u(t)\big),
	$$
	with some $f$, which has large negative values on $[0,a)$. Of course, the latter is impossible if $f(a)>0$, since $f$ must be continuous, but with suitable approximation argument it is possible to derive Theorem \ref{thm:wiener[a,b]} from Theorem~\ref{thm:wiener[0,T]}.
\end{remark}

\subsubsection{$0\in[a,b]$}

Now $a<0$, $b>0$, $\T_0=\{0\}$, $\T_1 = [a,0)\cup(0,b]$. The primary space is $C_0([a,b];\{0\}) = \set{f\in C([a,b]): f(0)=0}$ with dual $M([a,0)\cup(0,b])$.  

The covariance function is equal to $R(t,s) = t\wedge s$, for $t,s>0$, $- (t\vee s)$ for $t,s<0$, and $0$ if $ts\le 0$. Then for $t\ge 0$ the covariance operator is 
\begin{equation*} %\label{eq:cov-wiener-ab+}
\begin{gathered}
\mathcal R \mu(t) = \int_a^b R(t,s)\mu(ds) = \int_0^b \int_{0}^t \ind{[0,s]}(u) du \mu(ds) = \int_0^t \mu([u,b])du = \int_0^t J \mu (u)du,
\end{gathered}
\end{equation*}
where $J\mu(u) = \mu([u,b])$, $u\in [0,b]$. Similarly, for $t\in[a,0)$
\begin{equation*}
\mathcal R \mu(t) = \int_t^0 \mu ([a,u])du,
\end{equation*}
where $J\mu(u) = \mu([a,u])$, $u\in [a,0)$, and
\begin{gather*}
\sprod{\mathcal R \mu,\nu} = \big(J\mu,J\nu\big)_{L^2[a,b]}.
\end{gather*}

Consequently, the operator  $J$ extends to an isomorphism between $\mathcal H_X$ and $L^2[a,b]$ and 
$$\mathbb H_X = \set{f\in AC([a,b]): f(0)=0, f'\in L^2[a,b]}.$$ 
As a result, we get a similar situation as for $[0,b]$. The difference is that now the function $\tilde f$ is non-decreasing and concave on $[0,b]$ but non-increasing and concave on $[a,0]$, so it can be ``glued'' together from the least non-decreasing concave majorant of $f$ on $[0,b]$ and the least non-increasing concave majorant on $[a,0]$. The bounds and the asymptotic behavior we obtain are  similar to the previous statements, so we skip the formulation. The important fact we should mention is that the values of $W$ on $[a,0]$ and $[0,b]$ are independent, so we can write
\begin{gather*}
\pr(W+f\le u \text{ on }[a,b]) = \pr(W+f\le u \text{ on }[a,0])\cdot \pr(W+f\le u \text{ on }[0,b])
\end{gather*}
and apply Theorem~\ref{thm:wiener[0,T]}. The results will agree with those obtained by direct application of the general theory, since  
$$
\bnorm{\tilde f'}^2_{L^2[a,b]} = \bnorm{\tilde f'}^2_{L^2[a,0]} + \bnorm{\tilde f'}^2_{L^2[0,b]}
$$
and 
$$
\int_a^b u(t)d\bigl(-\tilde f'(t)\bigr) = \int_a^0 u(t)d\bigl(-\tilde f'(t)\bigr) + \int_0^b u(t)d\bigl(-\tilde f'(t)\bigr).
$$

\subsection{Brownian bridge}

For convenience in this example we work with  $\mathbb{T}=[0,1]$. Let $X_t = B_t^0:= W_t - tW_1$, $t\in[0,1]$, be a Brownian bridge, which is a centered Gaussian process with covariance function $R(t,s) = \min(t,s)-ts$. The primary space is now 
$$C_{0,0}([0,1]) = \set{x\in C([0,1]): x(0)=x(1)=0},$$
with the dual space $M((0,1))$. Further, similarly to \eqref{eq:cov-wiener} the  covariance operator  is given by 
\begin{equation}\label{eq:cov-bb}
\begin{gathered}
\mathcal R \mu (t) = \int_0^1 \bigl(\min(t,s) - ts\bigr)\mu(ds) = \int_0^t \mu([s,1)) ds -  t \int_0^1 \int_0^s du\,\mu(ds) =\\
\int_0^t \mu([s,1)) ds -  t \int_0^1 \mu([u,1))du = \int_{0}^t \left(\mu((s,1)) - \int_0^1 \mu([u,1))du \right)ds = \int_{0}^t J_0 \mu(s) ds,
\end{gathered}
\end{equation}
where $J_0 \mu(s) = \mu([s,1)) - \int_0^1 \mu([u,1))du$.
Using simple transformations, we obtain 
\begin{gather*}
\sprod{\mathcal R \mu,\nu} = \int_0^b J_0\mu(u)J_0\nu(u)du = \big(J_0\mu,J_0\nu\big)_{L^2[0,1]}.
\end{gather*}
Consequently,   $J_0$ extends to an isometry between $\mathcal H_X$ and the completion of image of $J_0$ in $L^2[0,1]$, which easily seen to be 
$$L_0^2[0,1]:= \set{f\in L^2[0,1]: \int_0^1 f(t)dt=0}.$$ 
Hence, in view of \eqref{eq:cov-bb}, the Cameron--Martin space $\mathbb H_X = \mathcal R \mathcal H_X$ consists of absolutely continuous functions having square integrable derivative and vanishing at 0 and 1, which agrees with  the well known description of this RKHS, see e.g.\ \cite[Example 4.9]{lifshits}. Similarly to the previous example, the drift $\tilde f$ from Theorem 2.9 should satisfy $\tilde f\ge f$ and 
$$
\big(f'-\tilde f',\tilde f'\big)_{L^2[0,1]}\ge 0. 
$$
By Lemma~\ref{lem:concmajbb}, this is true for the least concave majorant of $f$, which is also a solution to~\eqref{eq:minprob}. So again we reproduce the known results for Brownian bridge, see \cite{BH05,BHHM03,BMHH03}.
\begin{theorem}
	Let $u\colon [0,1]\to \R$ be lower semicontinuous,  $f\in AC([0,1])$ be such that $f(0)=f(1) =0$, $f'\in L^2[0,1]$, $\tilde f$ be the least concave majorant of $f$, and $\tilde f'_-$ be its left derivative (right derivative at $0$).
	
	1. The probability $P_{f,u} = \pr(\forall t\in[0,1]\ B_t^0+f(t)\le u(t))$ admits the upper bound
	\begin{equation*} 
	P_{f,u} \le P_{f-\tilde f,u} \exp \Bigl \{- \frac{1}{2}\bnorm{\tilde f'_-}_{L^2[0,1]}^2 + \int_0^1 u(t) d\bigl(-\tilde f'_-(t)\bigr)\Bigr \}.
	\end{equation*}
	
	2. For any $u_{-}\in C([0,1])$ such that $u_{-}(t) < u(t)$ for all $t\in [0,1]$ and 
	$$
	P_{0,u,u_-} := \pr\big(\forall t\in[0,1]\ \ u_{-}(t)\le B^0_t\le u(t)\big)>0,
	$$
	the probability $P_{f,u}$ admits the lower bound 
	\begin{equation*}
	\begin{gathered}
	P_{f,u} \ge P_{0,u,u_{-}}  \exp \Bigl  \{ - \frac{1}{2} \bnorm{ \tilde f'_-}_{L^2[0,1]}^2 + \int_0^1 u_-(t) d\bigl(-\tilde f'_-(t)\bigr) \Bigr \}.
	\end{gathered}
	\end{equation*}
	
	3. If $u(0),u(1)>0$, then we have  
	$$
	\CE{\log}\,  P_{\cY f,u}\sim - \frac{\cY ^2}{2} \bnorm{ \tilde f'_-}_{L^2[0,1]}^2 + \cY \int_0^1 u(t) d\bigl(-\tilde f'_-(t)\bigr)+o(\cY), \quad \cY \to+\infty. 
	$$
	
\end{theorem}

\subsection{Brownian motion on $[0,+\infty)$}
Let $X = W$, $\T = [0,+\infty)$. Now $\T$ is locally compact, so we should use the ideas of Subsection~\ref{sec:localcomp}. But first we transform the parameter space conveniently, setting 
$$
Y_t = W_{t/(1-t)}, t\in[0,1).
$$
Now we should multiply $Y$ by some positive function $v \in  C([0,1))$ so that 
$v(t)Y_t \to 0$, $t\to 1-$. It is not hard to see that $v(t) = 1-t$ works. As a result, we can write
$$
\pr\bigl(\forall t\ge 0\ W_t + f(t)\le u(t)\bigr) = \pr\bigl(\forall t\in [0,1]\ Z_t +\bar f(t)\le \bar u(t)\bigr),
$$
where 
$$Z_t = (1-t)W_{t/(1-t)}, \quad  \bar f(t) =(1-t) f\bigl(t/(1-t)\bigr),  \quad \bar u(t) =(1-t) u\bigl(t/(1-t)\bigr).$$
It appears that the process $Z$ is a Brownian bridge on $[0,1]$, so we reduce the problem to the previous example; the solution $\tilde f$ to the constrained optimization problem is now a least non-decreasing concave majorant, as in Example~\ref{ex:wiener[0,b]} (see also \cite[Lemma 5.1]{BHHM05}).

\subsection{Volterra process}
Consider  $X_t = \int_0^t K(t,s) dW_s$, $t\in [0,T]$, where the Volterra kernel $K$  is such that $\sup_{t\in[0,T]} \int_0^t K(t,s)^2 ds<\infty$ and $X$ has continuous sample paths. 
In this case for any finite signed measure $\mu$ on $[0,T]$
\begin{gather*}
\mathcal R \mu(t) = \int_0^T R(t,s) \mu(ds) = \int_0^T \int_0^{t\wedge s} K(t,u) K(s,u) du\, \mu(ds) \\
= \int_0^t K(t,u) \int_u^T K(s,u)\mu(ds) du.
\end{gather*}
Consequently,  the covariance operator admits the following decomposition 
$\mathcal R = \mathcal K\mathcal K^*$, where 
$$
\mathcal K f(t) = \int_0^t K(t,s)f(s)ds,\quad 
\mathcal K^* \mu (s) = \int_s^T K(t,s)\mu(dt).
$$
Moreover,  we have 
\begin{gather*}
\sprod{\mu,\nu} = \int_0^T \int_0^T R(t,s) \mu(ds)\nu(ds) = \int_0^T \int_0^T \int_0^{t\wedge s} K(t,u) K(s,u) du\, \mu(ds)\nu(ds)\\
= \int_0^T \int_u^T K(t,u)\mu(dt)\int_u^T K(s,u)\nu(ds) du = \bprod{\mathcal K^* \mu, \mathcal K^* \nu}^2_{L^2[0,T]}.
\end{gather*}
As a result, $\mathcal H_X$ can be identified with a preimage of $L^2[0,T]$ under $\mathcal K^*$, and $\mathbb H_X$, with the image of $L^2[0,T]$ under $\mathcal K$. Despite the seemingly clear, as in the previous examples, description of the Cameron--Martin space, it is in general hard  to identify  the solution of the minimization problem \eqref{eq:minprob}. (See, for example, the article \cite{HMS15}, which considers the boundary non-crossing probabilities for fractional Brownian motion, in particular Theorem 3.1 and Corollary 3.2 therein.) Of course, there is a viable case contained in Corollary~\ref{cor:positive}: for any $\gamma\in M([0,T])$ and $f=\mathcal R \gamma$, the asymptotic expansion \eqref{eq:P_cf-asymp} holds, however, the Volterra structure does not really help here. 

\subsection{Brownian sheet}

Let $X$ be a Brownian sheet, i.e.\ a centered Gaussian process indexed by $\T = [0,T]^2$ and having the covariance function
$$
R\bigl((t_1,t_2),(s_1,s_2)\bigr) = \min(t_1, s_1)\cdot \min(t_2, s_2), (t_1,t_2), (s_1,s_2)\in \T. 
$$
Now $\T_0=(\{0\}\times[0,T])\cup ([0,T]\times \{0\})$, $\T_1 = \T\setminus \T_0 = (0,T]^2$. The primary space is $C_0([0,T]^2;\T_0) = \set{x\in C([0,T]): \forall t\in[0,T]\ x(0,t) = x(t,0)=0}$ with dual $M((0,T]^2)$, the space of finite signed measures on $(0,T]^2$. Similarly to Example~\ref{ex:wiener[0,b]}, the covariance operator is
\begin{equation*} %\label{eq:cov-wiener}
\begin{gathered}
\mathcal R \mu(t_1,t_2) = \int_0^T\int_0^T \min (t_1,s_1)\min (t_2,s_2) \mu(ds_1,ds_2) \\
%=  \int_0^T\int_0^T \int_0^{t_1} \ind{[0,s_1]}(u_1)du_1\int_0^{t_1} \ind{[0,s_2]}(u_2)du_2\, \mu(ds_1,ds_2) 
%\\ = \int_0^{t_1}\int_0^{t_2} \int_{u_1}^T\int_{u_2}^T \mu(ds_1,ds_2)
= \int_0^{t_1}\int_0^{t_2} \mu\bigl([u_1,T]\times[u_2,T]\bigr)du_2 du_1 = \int_0^{t_1}\int_0^{t_2} J_2\mu (u_1,u_2)du_2 du_1,
\end{gathered}
\end{equation*}
where $J_2\mu (u_1,u_2) = \mu\bigl([u_1,T]\times[u_2,T]\bigr)$, 
and
\begin{gather*}
\sprod{\mathcal R \mu,\nu} = \big(J_2\mu,J_2\nu\big)_{L^2\left([0,T]^2\right)},
\end{gather*}
so  $J_2$ extends to an isomorphism between $\mathcal H_X$ and $L^2\bigl([0,T]^2\bigr)$. Therefore, $\mathbb H_X$ consists of functions of the form $\int_0^{t_1}\int_0^{t_2} h(u_1,u_2)du_1\,du_2$, where $h\in L^2\bigl([0,T]^2\bigr)$, so again we get the well-known description of the Cameron-Martin space of Brownian sheet. 

We do not know the solution to the optimization problem \eqref{eq:minprob} in general, only in two particular case. The first case is where $f$ is itself the solution, then we have an ad hoc version of Corollary~\ref{cor:positive}. 
\begin{theorem} %\label{thm:wiener[a,b]}
	Let $X$ be a Brownian sheet, $u\colon [0,T]^2\to \R$ be a lower semicontinuous function such that $u(0,t)>0$ and $u(t,0)>0$ for all $t\in[0,T]$, $\gamma$ be a finite non-negative measure on $(0,T]^2$, and 
	$$
	f(t_1,t_2) = \int_0^{t_1}\int_0^{t_2} J_2\gamma(u_1,u_2)du_2\, du_1 = \int_0^{t_1}\int_0^{t_2} \gamma\bigl([u_1,T]\times[u_2,T]\bigr)du_2 du_1.
	$$
	Then, the following asymptotics holds:
	\begin{gather*}
	\log\, \pr\bigl( \forall t_1,t_2\in[0,T]\ \ X_{t_1,t_2}+\cY f(t_1,t_2)\le u(t_1,t_2)\bigr)\\
	= - \frac{\cY^2}{2} \bnorm{ J_2 \gamma}_{L^2\left([0,T]^2\right)}^2 + \cY \int_0^T\int_0^T u(t_1,t_2) \gamma(dt_1,dt_2) + o(\cY), \quad \cY\to+\infty. 
	\end{gather*}
\end{theorem}

The second case is $f(t_1,t_2) = f_1(t_1)\cdot f_2(t_2)$ with non-negative $f_1,f_2$ belonging to the RKHS of Wiener space, i.e., $f_i = \int_0^t h_i(s) ds$ with $h_i \in L^2[0,T]$, $i=1,2$. In this case, the solution to the optimization problem is $\tilde f(t_1,t_2) = \tilde f_1(t_1)\cdot \tilde f_2(t_2)$, where $\tilde f_i$ is the smallest non-decreasing concave majorant of $f_i$, $i=1,2$. Indeed, $\tilde f\ge f$ and, thanks to Lemma~\ref{thm:concmaj} 
$$
\int_0^T \tilde f_i'(s) \bigl( f_i'(s) - \tilde f_i'(s)\bigr) ds = 0,\quad i=1,2,
$$
hence
\begin{gather*}
\int_0^T\int_0^T f_1'(s_1)f_2'(s_2) \tilde f_1'(s_1)\tilde f_2'(s_2)  ds_1 ds_2 \\
= \int_0^T f_1'(s_1)\tilde f_1'(s_1)ds_1\int_0^T f_2'(s_2) \tilde f_2'(s_2)  ds_2 \\
= \int_0^T\tilde  f_1'(s_1)^2 ds_1\int_0^T\tilde  f_2'(s_2)^2 ds_2 =
\int_0^T\int_0^T \tilde f_1'(s_1)^2\tilde f_2'(s_2)^2  ds_1 ds_2,
\end{gather*}
equivalently,
\begin{gather*}
\int_0^T\int_0^T  \bigl( f_1'(s_1)f_2'(s_2) -  \tilde  f_1'(s_1)\tilde  f'_2(s_2)\bigr) \tilde f_1'(s_1)\tilde f_2'(s_2)  ds_1 ds_2 =0.
\end{gather*}
As in Example~\ref{ex:wiener[0,b]}, assuming that $f = \mathcal R \gamma$, $\tilde f = \mathcal R \tilde \gamma$, the last equality is equivalent to (G2). Thus, noting that $$\norm{\tilde \gamma}_{\HK} = \bnorm{\tilde f}_{L^2\left([0,T]^2\right)} = \bnorm{\tilde f_1}_{L^2[0,T]}\cdot \bnorm{\tilde f_1}_{L^2[0,T]},$$  we arrive at the following statement. 
\begin{theorem} %\label{thm:wiener[a,b]}
	Let $X$ be a Brownian sheet, $u\colon [0,T]^2\to \R$ be a lower semicontinuous such that $u(0,t)>0$ and $u(t,0)>0$ for all $t\in[0,T]$. Let also $f_1,f_2\in AC([0,T])$ be  non-negative functions such that $f_i(0) =0$ and $f_i' \in L^2[0,T]$, $i=1,2$, and $\tilde f_1,\tilde f_2$ be their least concave non-decreasing majorants. 	
	Then, 
	the following asymptotics holds as $\cY\to+\infty$:
	\begin{gather*}
	\log\, \pr\bigl( \forall t_1,t_2\in[0,T]\ \ X_{t_1,t_2}+\cY f_1(t_1)f_2(t_2)\le u(t_1,t_2)\bigr)\\
	= - \frac{\cY^2}{2} \bnorm{ \tilde f_1}_{L^2 [0,T]}^2\cdot \bnorm{ \tilde f_2}_{L^2[0,T]}^2 + \cY\int_0^T\int_0^T u(t_1,t_2) d\tilde f_1'(t_1) d\tilde f_2' (t_2) + o(\cY). 
	\end{gather*}
\end{theorem}

%\section{General Gaussian processes}
%\section{Appendix}
\appendix

\section{Auxiliary statements}

The following lemma summarizes properties of the least non-decreasing concave majorant. They are probably  well known, but here we write them for completeness.
% \CE{ presenting some more  (ncessary) details in their proofs.}% rigorous proofs)} and for the sake of having them at one place.

\begin{lemma}\label{thm:concmaj}
	For a function $f\in AC([0,T])$ with $f(0)=0$, its  least  non-decreasing concave majorant $\tilde f$ exists and is also absolutely continuous with $\tilde f(0) = 0$. Moreover, if $f'\in L^2[0,T]$, then $\tilde f' \in L^2[0,T]$ and
	$$
	\int_0^T \big(f'(s) - \tilde f'(s)\big)\tilde f'(s) ds  = 0, \quad  \tilde f= \operatorname*{argmin}_{g \in \mathbb H_W, g\ge f} \norm{g'}_{L^2[0,T]},
	$$
	where
	$$\mathbb H_W = \set{g:[0,T]\to \R: f(t) = \int_0^t h(s)ds, \quad h\in L^2[0,T]}$$
	is the RKHS of a standard Wiener process $W$.
\end{lemma}
\begin{proof}
	Let $t_0 = \operatorname{argmax}_{[0,T]} f$. The least non-decreasing concave majorant is non-decreasing on $[0,t_0]$ and constant on $[t_0,T]$ with $\tilde f(t_0) = f(t_0)$, so it is enough to prove the statement on $[0,t_0]$ given $t_0>0$. To simplify the notation, we will assume $t_0=T$.
	
	Since $\tilde f$ is non-decreasing on $[0,T]$ and exceeds $f$, it is not less than the least  non-decreasing majorant $\hat f(t) = \max_{s\in[0,t]} f(s)$ of $f$. Further, for all $x<y$, $\hat f(y)-\hat f(x)$ does not exceed the variation of $f$ on $x,y$, which is equal to $\int_x^y |f'(s)|ds$. Therefore, $\hat f$ is absolutely continuous with $|\hat f'(t)\le |f'(t)|$ a.e., in particular, $|\hat f'|$ is square integrable. Consequently, it is enough to prove the statement for a non-decreasing $f$ (equivalently, for non-negative $f'$).
	
	Let $h$ denote the monotone rearrangement of $f'$ (i.e.\ $h(t) = \sup\{x: \lambda(\{s\in[0,T]: f'(s)\ge x\})\le t\}$, $t\in[0,T]$). It is well-known that  $\int_0^T h(t)^2 dt = \int_0^T f'(t)^2 dt$ and for all $t\in[0,T]$
	\begin{equation}\label{ineq}
	\int_0^t h(s)ds \ge \int_0^t f'(s)ds = f(t).
	\end{equation}
	Since $h$ is non-increasing, $g(t):= \int_0^t h(s)ds$ is a non-decreasing concave majorant of $f$. Moreover, $g$ is continuous with $g(0) = 0$ and $g(T) = f(T)$. Therefore, $\tilde f$, being the least non-decreasing concave majorant, lies between $f$ and $g$, so it is also continuous at $0$ and $T$ with $\tilde f(0) = 0$, $\tilde f(T) = f(T)$. Since $\tilde f$ is non-decreasing, it can only have jump discontinuities, which, however, would condradict concavity, so it is continuous.
	
	Now let $Z = \{t\in[0,T]: \tilde f(t) = f(t)\}$. Since $f$ and $\tilde f$ are continuous, this set is closed with $\set{0,T}\subset Z$. Its complement is an open set, so it is a union of disjoint open intervals, say, $\bigcup_{n\ge 1} (a_n,b_n)$. Now for any $n\ge 1$, $\tilde f$ is affine on $[a_n,b_n]$ with $\tilde f(a_n)  = f(a_n)$, $\tilde f(b_n) = f(b_n)$. Therefore, denoting for any $n\ge1$ \ $\bar f_n = \frac{f(b_n)-f(a_n)}{b_n-a_n}$, we have  $$\int_{a_n}^{b_n}  \big( f'(s)- \tilde f'(s)\big)\tilde f'(s) ds = \bar f_n \int_{a_n}^{b_n} \big(f'(s) - \bar  f_n\big)ds = 0,
	$$
	whence
	\begin{gather*}
	\int_0^T \big(f'(s)- \tilde f'(s)\big)\tilde f'(s) ds = \int_Z  \big( f'(s) - \tilde f'(s)\big) \tilde f'(s) ds \\ + \sum_{n\ge 1} \int_{a_n}^{b_n} \big(f'(s) - \tilde f'(s)\big)\tilde f'(s)  ds = 0.
	\end{gather*}
	%Finally, note that
	Since for each $n\ge 1$, $\int_{a_n}^{b_n} \tilde f'(s)^2 ds\le \int_{a_n}^{b_n} f'(s)^2 ds$ by Jensen's inequality, then $\tilde f'\in L^2[0,T]$ follows.
	
	Further, since $\bnorm{\tilde f'}_{L^2[0,T] } \le \norm{f'}_{L^2[0,T] } $, the minimiser of $\norm{g}_{L^2[0,T]}$ for all $g\ge f, g \in H_W$ belongs to the set $A_g:=\{ g\ge \tilde f, g\in H_W\}$ and $g$ is concave, non-decreasing. Consequently, it belongs also to the set $A_g^*= \{ g\in A_g, g(T)= \tilde f(T)\}$. For any  $ g \in A_g^*$ we have
	$$ \int_0^T  \big( g'(s) - \tilde f'(s)\big)\tilde f'(s) ds = \int_0^T \big(g(s)- \tilde f(s)\big) d(- \tilde f'(s)) \ge 0,
	$$
	hence
	$$   \norm{g'}_{L^2[0,T] }^2 = \bnorm{g'-\tilde f'}_{L^2[0,T] }^2 + \bnorm{\tilde f'}_{L^2[0,T] }^2 +
	\int_0^T \big(g'(s)- \tilde f'(s)\big) \tilde f'(s) ds \ge  \bnorm{\tilde f'}_{L^2[0,T] }^2 $$
	and therefore  the minimizer is unique and equals $\tilde f$  establishing the claim.
\end{proof}

The following statement for Brownian bridge is proved similarly and therefore we omit its proof.
\begin{lemma}\label{lem:concmajbb}
	For an absolutely continuous function $f\colon [0,T]\to\R$, with $f(0)=f(T)=0$, its least concave majorant $\tilde f$ is also absolutely continuous with $\tilde f(0) = \tilde f(T) = 0$. Moreover, if $f'\in L^2[0,T]$, then $\tilde f' \in L^2[0,T]$ and
	$$
	\int_0^T \tilde f'(s)\big(\tilde f'(s)- f'(s)\big) ds  = 0, \quad  \tilde f= \operatorname*{argmin}_{g \in \mathbb H_{B^0}, g\ge f} \norm{g'}_{L^2[0,T]},
	$$
	where
	$$\mathbb H_{B^0} = \set{g:[0,T]\to \R: f(t) = \int_0^t h(s)ds, h\in L^2[0,T], \int_0^T h(t)dt =0}$$
	is the RKHS of a Brownian bridge $B^0$.
\end{lemma}

%\begin{lemma} Let $X_t, t\in \mathbb{T}$ be a continuous centered Gaussian process as in Introduction. For any lower semicontinuous function $u: \mathbb{T} \to \R$  and any $f \in  \mathbb{H}_X$ we have 
%		%	 $f,h\in \mathbb{H}_X$ and any measurable $S \subset E$ we have 
%		\begin{equation}
%		\label{eq:loweFA:b}
%		\log \Bigl( \frac{P_{f,u}}{P_{f- \tilde f,u}} \Bigr) \ge   
%		- \frac{1} 2 {\bnorm{\tilde f}_{\mathbb{H}_X}^2 }- \bnorm{\tilde f}_{\mathbb{H}_X} \sqrt{-2 \log P_{ f- \tilde f, u}}\,,  
%		\end{equation}	  
%		where  $\tilde f= \operatorname*{argmin}_{h \ge f, h\in \mathbb{H}_X} \norm{h}_{\mathbb{H}_X}$ is the projection of the zero function on the closed convex set $\{h \in \mathbb{H}_X, h\ge f\}$.  
%		\label{lem:frank}
%	\end{lemma}
%	\begin{proof} The 
%	
%	For any $f,h \in  \mathbb{H}_X$ we have for any measurable set $S\subset \R^T$ such that 
%		$\mathbb{P}( X + f- h \in S ) >0$ 
%		\begin{equation}
%		\label{eq:loweFA}
%		\log \frac{\pr( X + f \in S )}{\pr( X + f- h \in S )} \ge   
%		- \frac{1} 2 {\norm{h}_{\mathbb{H}_X}^2 }- \norm{h}_{\mathbb{H}_X} \sqrt{-2  \log \pr( X + f- h \in S ) },
%		\end{equation}	  	
%		which follows with the same arguments as in the proof of \cite[Proposition~1.6]{aurzada-dereich}  applying the reverse H\"older inequality with 
%		$$p= 1+ \norm{h}_{\mathbb{H}_X} \sqrt{ - 1/ (2\log \mathbb{P}( X + f- h \in S ))}> 1.$$
%		Hence the claim follows by Projection theorem.% and the fact that $ \tilde f \ge f$.    	
%	\end{proof}

\section*{Acknowledgements}
	%Yu. Mishura and G. Shevchenko thank E. Hashorva for his hospitality while their visit to University of Lausanne, during which this research was done.
	
	The authors thank anonymous referees for their careful reading of the manuscript and valuable remarks which helped to improve the article.
	Financial support from  SNSF Grant 200021-175752/1 is kindly acknowledged. 
%\end{acknowledgements}

\bibliographystyle{plain}

\bibliography{abcbib}
\end{document}